\documentclass[12pt]{article}
\usepackage[a4paper]{geometry}
\usepackage{amsthm,amssymb,amsmath}
\usepackage{graphicx,cite,color}
\usepackage{longtable,pdflscape,booktabs,caption,multicol}
\usepackage[colorlinks=true,citecolor=black,linkcolor=black,urlcolor=blue]{hyperref}

\theoremstyle{plain}
\newtheorem{theorem}{Theorem}
\newtheorem{lemma}[theorem]{Lemma}

\newtheorem{proposition}[theorem]{Proposition}
\newtheorem{claim}[theorem]{Claim}
\newtheorem{conjecture}[theorem]{Conjecture}

\theoremstyle{definition}

\theoremstyle{remark}
\newtheorem{remark}[theorem]{Remark}

\newcommand{\arxiv}[1]{\href{https://arxiv.org/abs/#1}{\texttt{arXiv:#1}}}

\newcommand{\N}{\mathbb{N}}
\newcommand{\Q}{\mathbb{Q}}

\newcommand{\Z}{\mathbb{Z}}

\DeclareMathOperator{\Circ}{Circ}

\title{On circulant nut graphs} 

\author{Ivan Damnjanovi\'c\thanks{The corresponding author.}\\
\small University of Ni\v s, Faculty of Electronic Engineering\\[-0.8ex] 
\small Aleksandra Medvedeva 14, 18115 Ni\v s, Serbia\\
\small\tt ivan.damnjanovic@elfak.ni.ac.rs\\
\and
Dragan Stevanovi\'c\thanks{The second author is supported by 
the Serbian Ministry of Education, Science and Technological Development through the Mathematical Institute of SASA,
and by the project F-159 of the Serbian Academy of Sciences and Arts.}\\
\small Mathematical Institute of the Serbian Academy of Sciences and Arts\\[-0.8ex]
\small Kneza Mihaila 36, 11000 Belgrade, Serbia\\
\small\tt dragan\_stevanovic@mi.sanu.ac.rs}

\begin{document}

\maketitle

\begin{abstract}
A nut graph is a simple graph
whose adjacency matrix has the eigenvalue~0 with multiplicity~1
such that its corresponding eigenvector has no zero entries.
Motivated by a question of Fowler et al.~[\emph{Disc. Math. Graph Theory} 40 (2020), 533--557]
to determine the pairs $(n,d)$ for which a vertex-transitive nut graph of order $n$ and degree $d$ exists,
Ba\v si\'c et al.\ [\arxiv{2102.04418}, 2021] initiated the study of circulant nut graphs.
Here we first show that the generator set of a circulant nut graph necessarily contains equally many even and odd integers.
Then we characterize circulant nut graphs with the generator set $\{x,x+1,\dots,x+2t-1\}$ for $x,t\in\N$, 
which generalizes the result of Ba\v si\'c et al.\ for the generator set $\{1,\dots,2t\}$.
We further study circulant nut graphs with the generator set $\{1,\dots,2t+1\}\setminus\{t\}$,
which yields nut graphs of every even order $n\geq 4t+4$
whenever $t$~is odd such that $t\not\equiv_{10}1$ and $t\not\equiv_{18}15$.
This fully resolves Conjecture~9 from Ba\v si\'c et al.~[ibid.].
We also study the existence of $4t$-regular circulant nut graphs for small values of~$t$,
which partially resolves Conjecture~10 of Ba\v si\'c et al.~[ibid.].

\end{abstract}

\section{Introduction}

Let $G=(V,E)$ be a simple graph on the vertex set $V=\{0,1,\dots,n-1\}$.
Graph $G$ is a \emph{nut graph} if its adjacency matrix $A_G$ has the eigenvalue~0 with multiplicity~1
such that an eigenvector corresponding to~0 has no zero entries.
We will assume that $G$ has at least two vertices,
as the one-vertex graph is trivially a nut graph.
Nut graphs were introduced as a mathematical curiosity
at the end of the 1990s in several papers of Sciriha~\cite{Sci1,Sci2,Sci3,Sci4} and Sciriha and Gutman~\cite{ScGu}.
Apparently, the name nut graph first appeared in~\cite{ScGu},
where it was also shown that nut graphs are characterized among singular graphs
as the graphs with a deck of nonsingular vertex-deleted subgraphs.
Further structural properties of nut graphs were later given in~\cite{Sci5,Sci6}.
The chemical justification for studying nut graphs of small degree
in terms of a single non-bonding molecular orbital and strong omni-conducting behaviour 
is given in~\cite{CoFG,FoPB,FPTBS,ScFo}.
Small nut graphs were enumerated by Fowler et al.~\cite{FPTBS},
while Coolsaet et al.~\cite{CoFG} devised a special algorithm for generating nut graphs.
Using this algorithm, they were able to enumerate all the nut graphs with up to 13~vertices,
nut graphs with the maximum degree at most three and up to 22~vertices,
cubic nut polyhedra with up to~34 vertices and nut fullerenes with up to 250~vertices.
Many further results about nut graphs are discussed at length in the monograph~\cite{ScFa}.

Structural properties of regular nut graphs were first studied by Sciriha and Fowler~\cite{ScFo}.
The existence of cubic and quartic nut graphs was considered by Gauci et al.~\cite{GaPS},
who fully characterized the orders for which such nut graphs exist:
a cubic nut graph of order~$n$ exists if and only if $n$ is an even integer, $n\geq 12$ and $n\notin\{14,16\}$,
while a quartic nut graph of order~$n$ exists if and only if $n\in\{8,10,12\}$ or $n\geq 14$.
Gauci et al.~\cite{GaPS} posed the problem to also determine orders for which $d$-regular nut graphs exist for $d\geq 5$, 
and Fowler et al.~\cite{FGGPS} managed to characterize such orders for all $d\leq 11$.
Fowler et al.~\cite{FGGPS} were also interested in further restricting nut graphs to be vertex-transitive,
and in this aspect they proved the following result.

\begin{theorem}[Fowler et al.~\cite{FGGPS}]
\label{th-vtnut}
Let $G$ be a vertex-transitive nut graph of order~$n$ and degree~$d$.
Then either $d\equiv_4 0$, $n\equiv_2 0$ and $n\geq d+4$,
or $d\equiv_4 2$, $n\equiv_4 0$ and $n\geq d+6$.
\end{theorem}

Using the recent catalogue of vertex-transitive graphs~\cite{HoRo},
they observed that a vertex-transitive nut graph exists for all pairs $(n,d)$ that satisfy the above conditions with $n\leq 42$.
Consequently, Fowler et al.~\cite[Question~11]{FGGPS} asked to characterize the pairs~$(n,d)$,
satisfying the requirements of Theorem~\ref{th-vtnut},
for which a vertex-transitive nut graph of order~$n$ and degree~$d$ exists.

Ba\v si\'c et al.~\cite{BaKS} approached this question by considering circulant nut graphs.
Recall that a graph~$G$ is {\em circulant} if its adjacency matrix~$A_G$ has the form
$$
A_G=
\begin{bmatrix}
a_0        & a_1 & a_2 & \dots & a_{n-1} \\
a_{n-1} & a_0 & a_1 & \dots & a_{n-2} \\
\vdots    & \vdots & \vdots &  & \vdots \\
a_1        & a_2 & a_3 & \dots & a_0
\end{bmatrix}.
$$
As we are interested in simple graphs,
this implies that $a_0=0$ and $a_j=a_{n-j}\in\{0,1\}$ for $j=1,\dots,n-1$.
A circulant graph may be more concisely described via the generator set~$S$
which contains the indices of those entries of the sequence $a_1,\dots,a_{\lfloor n/2\rfloor}$ that are equal to~$1$.
Hence for the given order~$n$ and the generator set~$S$,
we denote by $\Circ(n,S)$ the circulant graph with the vertex set $V=\{0,1,\dots,n-1\}$ and
the edge set
$$
E=\{(i,i\pm s)\colon i\in V, s\in S\},
$$
where $i\pm s$ is computed modulo~$n$.

Ba\v si\'c et al.~\cite{BaKS} proved that a 12-regular nut graph of order~$n$ exists if and only $n\geq 16$.
Thanks to the Fowler construction~\cite{FGGPS,GaPS}, 
which allows one to construct a $d$-regular nut graph of order~$n+2d$ from a $d$-regular nut graph of order~$n$,
it was sufficient to construct 12-regular nut graphs for $16\leq n\leq 39$.
For an even $n\in\{16,\dots,38\}$, 
the nut graphs were constructed as circulant graphs with the generator set either $\{1,\dots,6\}$ or $\{1,\dots,5,8\}$,
while for an odd $n\in\{17,\dots,39\}$, 
the nut graphs were obtained computationally by rewiring edges of circulant graphs.
They further proved the following result.

\begin{theorem}[Ba\v si\'c et al.~\cite{BaKS}]
\label{th-basic}
Let $n$ be even and $d$ divisible by~4.
Then $\Circ(n,\{1, \linebreak \dots,\frac d2\})$ is a nut graph if and only if
$\gcd(n, \frac d2+1)=1$ and $\gcd(\frac n2, \frac d4)=1$.
\end{theorem}

Ba\v si\'c et al.~\cite{BaKS} further conjectured that
a 12-regular circulant nut graph exists for every even $n\geq 16$ (\!\!\cite[Conjecture~9]{BaKS})
and, more generally, that
a $d$-regular circulant nut graph exists for every even $n\geq d+4$ and every~$d$ divisible by~$4$ (\!\!\cite[Conjecture~10]{BaKS}).

We are interested here in the properties of circulant nut graphs.
In the next section, we briefly review classical results on eigenvalues and eigenvectors of circulant graphs,
which we use to show that the generator set of a circulant nut graph 
necessarily contains equally many even and odd integers (Lemma~\ref{le-equally-many}).
In Section~\ref{sc-basic-generalization}
we characterize circulant nut graphs with the consecutive generator set $\{x,x+1,\dots,x+2t-1\}$ for $x,t\in\N$,
which generalizes Theorem~\ref{th-basic}.
In Section~\ref{sc-almost-consecutive}
we further study circulant graphs with the almost consecutive generator set $\{1,\dots,2t+1\}\setminus\{t\}$,
which yields nut graphs for all even $n\geq 4t+4$
whenever $t$ is odd and such that $t\not\equiv_{10} 1$ and $t\not\equiv_{18} 15$.
In particular, the case $t=3$ resolves Conjecture~9 of Ba\v si\'c et al.~\cite{BaKS}.
In Section~\ref{sc-other-generators} we further study the existence of $4t$-regular circulant nut graphs for small values of~$t$, 
which partially answers Question~11 of Fowler et al.~\cite{FGGPS} 
and also partially resolves Conjecture~10 of Ba\v si\'c et al.~\cite{BaKS}.

\section{Preliminaries}
\label{sc-preliminaries}

A classical result in linear algebra (see, e.g., \cite[Section~3.1]{Gray}) states that
the circulant matrix~$A$ with the first row $(a_0,\dots,a_{n-1})$ has the eigenvalues 
$$
P(1), P(\omega), \dots, P(\omega^{n-1}),
$$
where $\omega=e^{i \frac{2\pi}n}$ is the $n$-th root of unity and 
$$
P(y) = a_0 + a_1 y + a_2 y^2 + \cdots + a_{n-1} y^{n-1}.
$$
Moreover, for $j=0,\dots,n-1$,
$$
x_j = (\omega^j, \omega^{2j}, \omega^{3j}, \dots, \omega^{nj})
$$
is an eigenvector corresponding to the eigenvalue~$P(\omega^j)$.

Let $G=\Circ(n,S)$ be a circulant graph of order~$n$ with the generator set $S=\{s_0,\dots,s_{k-1}\}$,
with $1\leq s_0<\dots<s_{k-1}\leq n/2$.
We assume that $S$ is non-empty, as otherwise $G$ has no edges 
and all eigenvalues of its adjacency matrix are zeros.
Let $\lambda_0,\lambda_1,\dots,\lambda_{n-1}$ be the eigenvalues of~$A_G$.
Here $\lambda_0=P(1)$ is the degree of~$G$,
while for the remaining eigenvalues we have 
$$
\lambda_j=P(\omega^j)=P(\omega^{n-j})=\lambda_{n-j}
$$ 
for $j=1,\dots,n-1$.
Namely, 
if $s_{k-1} < \frac n2$, from the equality $\omega^{n-j} = \omega^{-j}$ we obtain
\begin{align}
\label{eq-prva}
\lambda_j 
         = P\left(\omega^j\right) 
       &= \left( \omega^{j s_0} + \frac{1}{\omega^{j s_0}} \right) 
         + \cdots 
         + \left( \omega^{j s_{k-1}} + \frac{1}{\omega^{j s_{k-1}}} \right), \\
\label{eq-druga}
\lambda_{n-j}
         = P\left(\omega^{n-j}\right) 
       &= \left( \omega^{j s_0} + \frac{1}{\omega^{j s_0}} \right) 
         + \cdots 
         + \left( \omega^{j s_{k-1}} + \frac{1}{\omega^{j s_{k-1}}} \right).
\end{align}
On the other hand, if $n$ is even and $s_{k-1} = \frac n2$, 
we have
\begin{align}
\label{eq-treca}
\lambda_j
         = P\left(\omega^j\right) 
       &= \left( \omega^{j s_0} + \frac{1}{\omega^{j s_0}} \right) 
         + \cdots 
         + \left( \omega^{j s_{k-2}} + \frac{1}{\omega^{j s_{k-2}}} \right) 
         + \omega^{j s_{k-1}}, \\
\label{eq-cetvrta}
\lambda_{n-j}
         = P\left(\omega^{n-j}\right) 
       &= \left( \omega^{j s_0} + \frac{1}{\omega^{j s_0}} \right) 
         + \cdots 
         + \left( \omega^{j s_{k-2}} + \frac{1}{\omega^{j s_{k-2}}} \right) 
         + \omega^{j s_{k-1}}.
\end{align}
The term $\frac{1}{\omega^{j s_{k-1}}}$ is omitted in (\ref{eq-treca}) and~(\ref{eq-cetvrta}) 
due to $n - s_{k-1} = \frac n2 = s_{k-1}$. 

The equality $\lambda_j=\lambda_{n-j}$ for $j=1,\dots,n-1$ establishes the following lemma.
\begin{lemma}
\label{le-even-order}
If $G=\Circ(n,S)$ is a nut graph and $n\geq 2$, 
then $n$ is even, $\lambda_{n/2}=P(\omega^{n/2}) = 0$ and 
$\lambda_j=P(\omega^j) \neq 0$ for all $j\in \{0, 1, \ldots, n-1\} \setminus \{n/2\}$.
\end{lemma}

\begin{proof}
Clearly, $\lambda_0=P(1) \neq 0$. 
For all $j \in \{1, 2, \ldots, n-1\}\setminus\{n/2\}$, 
we know that $n-j \in \{1, 2, \ldots, n-1\}\setminus\{n/2\}$ with $j \neq n-j$. 
If $\lambda_j=P\left(\omega^j\right)=0$, then $\lambda_{n-j}=P\left(\omega^{n-j}\right) = 0$, 
so that the kernel of the adjacency matrix~$A_G$ is at least $2$-dimensional,
and $G$ would not be nut.
Hence we must have $\lambda_j=P\left(\omega^j\right) \neq 0$ for all $j \neq n/2$. 

If $n$ is odd, then $n/2$ is not an integer, and so the kernel of $A_G$ is $0$-dimensional, and $G$ would again not be nut.
Hence $n$ must be even. 

Finally, we know that $G$ is nut, and
so it must have an eigenvalue equal to~0.
The only remaining possibility is that $\lambda_{n/2}=P\left(\omega^{n/2}\right)=0$.
\end{proof}

Due to $P(\omega^{n/2})=0$ and $P(\omega^j)\neq 0$ for $j\neq n/2$, 
the previous lemma implies that the eigenspace of~$0$ in a circulant nut graph
is spanned by the vector $x_{n/2}=(-1,1,\dots,-1,1)$.
This was also proved in~\cite[Lemma 6]{BaKS}, although in a different manner.

A more careful analysis of Eqs. (\ref{eq-prva})--(\ref{eq-cetvrta}) implies
a necessary condition for the generator set of a circulant nut graph.

\begin{lemma}
\label{le-equally-many}
If $G=\Circ(n,S)$ is a nut graph with $n\geq 2$,
then for some $t>0$ the set $S$ consists of $t$~odd and $t$~even integers from $\{1,\dots,n/2-1\}$.
Consequently, $G$ is $4t$-regular.
\end{lemma}

\begin{proof}
Let $S=\{s_0,\dots,s_{k-1}\}$ with $1\leq s_0<\dots<s_{k-1}\leq n/2$.
If $s_{k-1} = n/2$, then from Eq.~(\ref{eq-treca}) and Lemma~\ref{le-even-order} we have
$$
0 = P\left(\omega^{n/2}\right) 
   = \left(\omega^{s_0n/2} + \frac{1}{\omega^{s_0n/2}} \right) 
   + \cdots 
   + \left(\omega^{s_{k-2}n/2} + \frac{1}{\omega^{s_{k-2}n/2}} \right) 
   + \omega^{\frac n2 \cdot \frac n2},
$$
which by $\omega^{n/2}=-1$ reduces to
$$
0 = 2 \left[ (-1)^{s_0} + \cdots + (-1)^{s_{k-2}} \right] + (-1)^{n/2}.
$$   
The above equality cannot hold as the right-hand side is odd.
Hence $s_{k-1} < n/2$. 
From Eq.~(\ref{eq-prva}) we then have
$$
0 = P\left(\omega^{n/2}\right)
   = \left( \omega^{s_0n/2} + \frac{1}{\omega^{s_0n/2}} \right) 
   + \cdots 
   + \left( \omega^{s_{k-1}n/2} + \frac{1}{\omega^{s_{k-1}n/2}} \right)
$$
which reduces to
$$
0 = 2 \left[ (-1)^{s_0} + \cdots + (-1)^{s_{k-1}} \right].
$$
This equality implies that there are equally many odd and even elements in the set $\{s_0, s_1, \ldots, s_{k-1}\}$.
Hence $k=2t$ for some $t\in\N$ and the lemma follows. 
\end{proof}


\begin{remark}
\label{re-pet}
We note here that
if the generator set~$S$ of a circulant graph $G=\Circ(n,S)$
consists of $t$ odd and $t$ even integers from $\{1,\dots, n/2-1\}$,
then $G$ has $(-1,1,\dots,-1,1)$ as an eigenvector of the eigenvalue $0=P(\omega^{n/2})$.
Since this vector has no zero entries,
a necessary and sufficient condition for $G$ to be a nut graph
is that the eigenvalue~$0$ has multiplicity~$1$,
i.e., that $P(\omega^j)\neq 0$ for $j\neq n/2$.
Further, 
$P(1)=4t>0$, 
while $P(\omega^j)=P(\omega^{n-j})$ for $j=1,\dots, n-1$, implies that 
$G$ is a nut graph if and only if 
$P(\omega^j)\neq 0$ for each $j\in\{1, \ldots, n/2-1\}$.
We will rely extensively on this remark in the following sections.
\end{remark}

\section{Circulant nut graphs with a consecutive generator set}
\label{sc-basic-generalization}

Ba\v si\'c et al.~\cite{BaKS} showed that
the circulant graph $\Circ(n,\{1,2,\dots,d/2\})$, for even $n$ and $d\equiv_4 0$,
is a nut graph if and only if $\gcd(\frac d2+1,n)=\gcd(\frac d4, \frac n2)=1$.
In this section we generalize this result to circulant graphs
with the generator set $\{x,x+1,\dots,x+2t-1\}$.

\begin{theorem}
\label{th-consecutive}
Let $G=\Circ(n, \{x, x+1, \ldots, x+2t-1\})$ for $n, x, t\in\N$, 
such that $n$ is even and $n\geq 2x+4t$.
Then $G$ is a nut graph if and only if $\gcd(\frac n2,t)=\gcd(\frac n2, 2x+2t-1)=1$.
\end{theorem}

\begin{proof}
As the generator set consists of $2t$~consecutive integers,
it contains $t$ odd and $t$ even integers,
so in line with Remark~\ref{re-pet},
$G$ is a nut graph if and only if $P(\omega^j)\neq 0$ for $j\in\{1,\dots,\frac n2-1\}$.
    
Let $j\in\{1, \dots, n/2-1\}$ and $y=\omega^j$.
Then $P(\omega^j)=0$ is equivalent to
\begin{align*}
0 &=
   \left( y^x + \frac{1}{y^x} \right) 
+ \left( y^{x+1} + \frac{1}{y^{x+1}} \right) 
+ \cdots 
+ \left( y^{x+2t-1} + \dfrac{1}{y^{x+2t-1}} \right) \\
   &= \frac{1}{y^{x+2t-1}} + \cdots + \dfrac{1}{y^x} + y^x + \cdots + y^{x+2t-1}.
\end{align*}
Adding $\dfrac{1}{y^{x-1}} + \cdots + \dfrac 1y + 1 + y + \cdots + y^{x-1}$ to both sides, 
we get
\begin{align*}
&& \dfrac{1}{y^{x-1}} + \cdots + y^{x-1}
     &= \dfrac{1}{y^{x+2t-1}} + \cdots + y^{x+2t-1} \\
&\Leftrightarrow&
     \dfrac{1}{y^{x-1}} \left( 1 + y + \cdots + y^{2x-2} \right)
     &= \dfrac{1}{y^{x+2t-1}} \left( 1 + y + \cdots + y^{2x+4t-2} \right) .
\end{align*}
Since $1\leq j\leq n/2-1$, we have $y=\omega^j\neq 1$,
and we further obtain the equivalences
\begin{align*}
&& \dfrac{1}{y^{x-1}} \cdot \dfrac{y^{2x-1} - 1}{y-1}
  &= \dfrac{1}{y^{x+2t-1}} \cdot \dfrac{y^{2x+4t-1} - 1}{y-1} \\
&\Leftrightarrow&
y^{2t} \left( y^{2x-1} - 1 \right)
  &= y^{2x+4t-1} - 1 \\
&\Leftrightarrow&
0 &= y^{2x+4t-1} - y^{2x + 2t - 1} + y^{2t} - 1 \\
&\Leftrightarrow&
0 &= \left( y^{2t} - 1 \right) \left( y^{2x + 2t - 1} + 1 \right).
\end{align*}
Since $y = \omega^j$, 
we have that $y^{2t} - 1 = 0$ is equivalent to $n\mid2jt$, i.e., to $\frac n2 \mid jt$ as $n$ is even. 
Let $b\!=\!\gcd(\frac n2,t)$.
If $b>1$, then $j=\frac n{2b}$ is a positive integer smaller than $\frac n2$.
Hence $jt = \frac n{2b}t = \frac n2 \cdot \frac tb$, where $\frac tb$ is also an integer, so that $\frac n2\mid jt$. 
This implies that $P(\omega^j)=0$ for $j<\frac n2$ and hence $G$ is not a nut graph.
On the other hand,
if $\gcd(\frac n2,t)=1$, then $\frac n2\mid jt$ implies $\frac n2\mid j$, 
which is not possible for $j\in\{1,\ldots, \frac n2 -1\}$, 
and so $P(\omega^j)\neq 0$ for $j<\frac n2$.

Similarly, $y^{2x+2t-1} = \omega^{(2x+2t-1)j}$ implies that 
$y^{2x+2t-1} + 1 = 0$ is equivalent to $2x + 2t - 1$ being a multiple of~$\frac n2$, but not of~$n$.
Let $b=\gcd(\frac n2, 2x+2t-1)$.
Since $2x+2t-1$ is odd, $b$ is also odd. 
If $b>1$, then $j=\frac n{2b}$ is a positive integer smaller than $\frac n2$.
Hence $(2x+2t-1)j = (2x+2t-1) \frac n{2b} = \frac{2x+2t-1}{b} \cdot \frac n2$,
and since $\frac{2x+2t-1}{b}$ is odd,
$(2x+2t-1)j$ is a multiple of $\frac n2$, but not of~$n$.
Hence $P(\omega^j)=0$ for $j<\frac n2$ and $G$ is not a nut graph.
On the other hand, 
if $\gcd(\frac n2, 2x+2t-1)=1$
then $\frac n2\mid(2x+2t-1)j$ implies $\frac n2\mid j$, 
which is not possible for $j\in\{1,\dots,\frac n2 -1\}$,
and so $P(\omega^j)\neq 0$ for $j<\frac n2$.

To conclude, if either $\gcd(\frac n2, t)>1$ or $\gcd(\frac n2, 2x+2t-1)>1$,
then $P(\omega^j)=0$ for some $j\in\{1,\dots, \frac n2 -1\}$, 
implying that $G$ is not a nut graph.
On the other hand, if $\gcd(\frac n2, t)=\gcd(\frac n2, 2x+2t-1)=1$,
then $P(\omega^j)\neq 0$ for each $j\in\{1,\dots,\frac n2 -1\}$,
so that $G$ is a nut graph in this case.
\end{proof}

\section{Circulant nut graphs with almost consecutive generator sets}
\label{sc-almost-consecutive}

Ba\v si\'c et al.~\cite[Conjecture~9]{BaKS} conjectured that
for every even $n\geq 16$ there exists a 12-regular circulant nut graph $\Circ(n,\{a_1,\dots,a_6\})$.
Our computational results actually reveal that the single generator set $S_3=\{1,2,4,5,6,7\}$ 
provides a positive answer to this conjecture for all even $n\geq 16$.
For this particular generator set, we have
$$
P(y) = \left(y+\frac1y\right)+\left(y^2+\frac1{y^2}\right)+\left(y^4+\frac1{y^4}\right)
       + \left(y^5+\frac1{y^5}\right)+\left(y^6+\frac1{y^6}\right)+\left(y^7+\frac1{y^7}\right).
$$
By Remark~\ref{re-pet},
$\Circ(n,\{1,2,4,5,6,7\})$ for an even $n\geq 16$ is a nut graph 
if and only if 
$$
P(\omega^j) = P(e^{2 i\pi\frac jn})\neq 0\quad\mbox{ for each }j\in\{1,\dots,n/2-1\}.
$$
Since $\omega^j\neq0$, the above condition is equivalent to
\begin{equation}
\label{eq-P*}
P^*(e^{2i\pi\frac jn})\neq 0\quad\mbox{ for each }j\in\{1,\dots,n/2-1\},
\end{equation}
where 
$$
P^*(y) = y^7P(y)=y^{14}+y^{13}+y^{12}+y^{11}+y^9+y^8+y^6+y^5+y^3+y^2+y+1
$$
is obtained by multiplying~$P(y)$ with $y^{\max(S_3)}$.

Recall now that the $n$-th cyclotomic polynomial~$\Phi_n(y)$ (see, e.g., \cite{cyclo}) is 
the minimal polynomial in~$\Q[y]$ for the primitive roots of unity $e^{2i\pi\frac jn}$ with $\gcd(j,n)=1$.
$\Phi_n(y)$ is the unique irreducible polynomial with integer coefficients 
that divides $y^n-1$, but does not divide $y^k-1$ for any $k<n$, and
may be computed recursively via
\begin{equation}
\label{eq-cyclotomic}
\Phi_n(y) = \frac{y^n-1}{\prod_{d\mid n, 1\leq d<n} \Phi_d(y)}
\end{equation}
(from where $\Phi_1(y)=y-1$, $\Phi_2(y)=y+1$, $\Phi_3(y)=y^2+y+1$, \dots).
Conversely, the M\"obius inversion formula yields
\begin{equation}
\label{eq-cyclotomic-2}
\Phi_n(y) = \prod_{d\mid n} (y^d-1)^{\mu(n/d)}.
\end{equation}

Hence if $P(e^{2i\pi\frac jn})=0$ for some $j\in\{1,\dots,n/2-1\}$,
then the primitive root of unity $e^{2i\pi\frac ab}$ for $a=\frac j{\gcd(j,n)}$, $b=\frac n{\gcd(j,n)}$ is the root of $P^*(y)$,
and consequently 
$$
\Phi_b(y)\mid P^*(y).
$$
The degree of $\Phi_b(y)$ is $\phi(b)$, the Euler's totient function,
so $\Phi_b(y)\mid P^*(y)$ also implies that 
$$
\phi(b)\leq\deg P^*(y)=14.
$$
Hence to check that $P^*(y)$ has no root of the form $e^{2i\pi\frac jn}$ with $1\leq j\leq n/2-1$,
and consequently, that $\Circ(n,\{1,2,4,5,6,7\})$ is a nut graph for each even $n\geq 16$,
it suffices to check that $P^*(y)$ is not divisible by~$\Phi_b(y)$ for any~$b\geq 3$ such that $\phi(b)\leq14$.
(Since $j<n/2$ we have $b=\frac n{\gcd(j,n)}>\frac n{n/2}$.)
We can use the lower bound~\cite[Theorem~15]{RoSc}
\begin{equation}
\label{eq-phi-lower-bound}
\frac b{e^\gamma \log\log b + \frac{2.51}{\log\log b}} < \phi(b),
\end{equation}
valid for $b\geq 3$, where $\gamma=0.57721\dots$ is the Euler's constant,
to determine all integers~$b$ such that $\phi(b)\leq 14$.
Namely, in such case $\frac b{e^\gamma \log\log b + \frac{2.51}{\log\log b}} < 14$,
and a simple iterative method for approximating the fixed point of
$f_{14}(b) = 14\left(e^\gamma \log\log b + \frac{2.51}{\log\log b}\right)$
yields in a handful of iterations that 
$$
\phi(b)\leq 14\quad\Rightarrow\quad b\leq 60.
$$

The following table lists the remainders obtained after dividing~$P^*(y)$ by the cyclotomic polynomials~$\Phi_b(y)$ 
for those values $3\leq b\leq 60$ for which indeed $\phi(b)\leq 14$.
\begin{center}
{\footnotesize
\begin{tabular}{rl}
\toprule 
$b$ & $P^*(y) \pmod{\Phi_b(y)}$ \\ 
\midrule
3 & $-3y$ \\
4 & $2y$ \\
5 & $y^{3} + y$ \\
6 & $3y$ \\
7 & $-y^{4} - y^{3}$ \\
8 & $2y^{3} - y^{2} + 1$ \\
9 & $y^{5} + y^{4} + y^{3} + y^{2} + y + 1$ \\
10 & $y^{3} + 2y^{2} + y$ \\
11 & $y^{9} + y^{8} + y^{6} + y^{5} + 2y^{3} + 2y^{2} + 2y + 2$ \\
12 & $y^{2} + 2y + 1$ \\
13 & $-y^{10} - y^{7} - y^{4} + y + 1$ \\
14 & $-y^{4} + y^{3} + 2$ \\
15 & $-y^{7} + y^{5} - y^{4} + 1$ \\
16 & $-y^{4} + y^{2}$ \\
18 & $y^{5} - y^{4} + y^{3} - y^{2} + y - 1$ \\
20 & $y^{7} + 2y^{6} - 2y^{4} + y^{3} + y^{2} - y$ \\
21 & $3y^{11} - y^{10} + 3y^{8} - 2y^{7} - y^{6} + 2y^{5} + y^{4} - y^{3} + 2y^{2} + 2y - 2$ \\
22 & $y^{9} + y^{8} + y^{6} + y^{5}$ \\
24 & $y^{7} + y^{6} + 2y^{5} + y^{4} - y - 1$ \\
26 & $2y^{11} - y^{10} + 2y^{9} + y^{7} + 2y^{5} - y^{4} + 2y^{3} + y - 1$ \\
28 & $2y^{11} + y^{10} + y^{7} + 2y^{6} - y^{4} + 2y^{3} + 2y^{2} - 1$ \\
30 & $y^{7} + 4y^{6} + 3y^{5} + y^{4} - 2y^{2} - 2y + 1$ \\
36 & $y^{11} + y^{9} + 2y^{8} + y^{7} + 2y^{6} + y^{5} + y^{3}$ \\
42 & $y^{11} + y^{10} + 2y^{9} + y^{8} + y^{6} + 2y^{5} + y^{4} + y^{3}$ \\
\bottomrule
\end{tabular}
}
\end{center}
Since $P^*(y)$ is apparently not divisible by any of these cyclotomic polynomials, we arrive at the following proposition.

\begin{proposition}
The circulant graph $\Circ(n,\{1,2,4,5,6,7\})$ is a nut graph for each even $n\geq 16$.
\end{proposition}

The generator set $S_3=\{1,2,4,5,6,7\}$ is a particular instance of the general generator set 
$$
S_t=\{1,\dots,2t+1\}\setminus\{t\},
$$
that provides many further examples of circulant nut graphs.
However, let us first showcase the opposite (and simpler) situation when $S_t$ does not yield a circulant nut graph.
It is obvious that for an even~$t$ the generator set~$S_t$ contains more odd than even entries,
and then $\Circ(n,S_t)$ is not a nut graph by Lemma~\ref{le-equally-many} for any even $n\geq 4t+4$.
Two further cases are covered by the following lemma.

\begin{lemma}
\label{le-2t+1-t}
The circulant graph $\Circ(n,S_t)$ for an even $n\geq 4t+4$ is not a nut graph if either:
a) $t\equiv_{10} 1$ and $5 \mid n$, or 
b) $t\equiv_{18} 15$ and $9 \mid n$.
\end{lemma}

\begin{proof}
For a given~$t$ we have
$$
P_{S_t}(y) = \left(y+\frac1y\right)+\cdots+\left(y^{t-1}+\frac1{y^{t-1}}\right)
                   +\left(y^{t+1}+\frac1{y^{t+1}}\right)+\cdots+\left(y^{2t+1}+\frac1{y^{2t+1}}\right),
$$
so that 
\begin{align*}
P^*_{S_t}(y) 
= y^{\max(S_t)}P_{S_t}(y)
= & \left(y^{4t+2} + \cdots + y^{3t+2}\right) + \left(y^{3t} + \cdots + y^{2t+2}\right) \\
+ & \left(y^{2t} + \cdots + y^{t+2}\right) + \left(y^t + \cdots + 1\right) \\
= & \ \frac{y^{4t+3} - y^{3t+2} + y^{3t+1} - y^{2t+2} + y^{2t+1} - y^{t+2} + y^{t+1} - 1}{y-1}.
\end{align*}
    
Suppose first that $t\equiv_{10} 1$ and $5 \mid n$ for some even $n \ge 4t+4$.
Then $\frac n{10}$ is an integer, and for $y=\omega^{n/10}$ we have $y^5 = -1$ and $y^{10} = 1$. 
Let $t = 10t'+1$. 
Then
\begin{align*}
(y-1)P^*_{S_t}(y) 
  &= y^{40t'+7} - y^{30t'+5} + y^{30t'+4} - y^{20t'+4} + y^{20t'+3} - y^{10t'+3} + y^{10t'+2} - 1\\
  &= y^7 - (-1) + y^4 - y^4 + y^3 - y^3 + y^2 - 1\\
  &= -y^2 + 1 + y^2 - 1\\
  &= 0,
\end{align*}
so that $P_{S_t}(\omega^{n/10})=0$. 
Then $\Circ(n,S_t)$ is not a nut graph by Lemma~\ref{le-even-order}.

Next, suppose that $t\equiv_{18} 15$ and $9 \mid n$ for some even $n\geq 4t+4$.
Then $\frac n9$ is an integer, and for $y=\omega^{n/9}$ we have $y^9=1$.
Let $t=18t'+15$.
Then 
\begin{align*}
(y-1)P^*_{S_t}(y) 
  &= y^{72t'+63} - y^{54t'+47} + y^{54t'+46} - y^{36t'+32} + y^{36t'+31} - y^{18t'+17} + y^{18t'+16} - 1\\
  &= 1 - y^2 + y - y^5 + y^4 - y^8 + y^7 - 1\\
  &= (y-y^2)(1+y^3+y^6)
\end{align*}
From $y^9-1=0$ we have $(y^3-1)(1+y^3+y^6) = 0$. 
Since $y^3 = \omega^{3n/9} = -\frac 12 + i\frac{\sqrt{3}}2 \neq 1$, it follows that $1+y^3+y^6 = 0$. 
Hence also $P_{S_t}(\omega^{n/9})=0$,
so that $\Circ(n,S_t)$ is not a nut graph by Lemma~\ref{le-even-order}.
\end{proof}

Interestingly,
$t\equiv_{10} 1$ and $t\equiv_{18} 15$ are the only two cases
when the generator set $S_t$ does not yield a circulant nut graph for all even $n\geq 4t+4$.
In the rest of this section 
we prove our main theorem.
\begin{theorem}
\label{th-main}
For each odd $t\geq 3$ such that $t\not\equiv_{10} 1$ and $t\not\equiv_{18} 15$,
the circulant graph $\Circ(n,S_t)$ is a nut graph for each even $n\geq 4t+4$.
\end{theorem}

The proof of this theorem depends on a number of technical lemmas and
the following theorem on divisibility of lacunary polynomials by cyclotomic polynomials.
\begin{theorem}[Filaseta, Schinzel\cite{FiSch}]
\label{th-FS}
Let $P(y)\in\Z[y]$ have $N$ nonzero terms and let $\Phi_b(y)\mid P(y)$.
Suppose that $p_1,\dots,p_k$ are distinct primes such that
\begin{equation}
\label{eq-FS-condition}
\sum_{j=1}^k (p_j-2)>N-2.
\end{equation}
Let $e_j$ be the largest exponent such that $p_j^{e_j}\mid b$.
Then for at least one~$j$, $1\leq j\leq k$, 
we have $\Phi_{b'}(y)\mid P(y)$, where $b'=b/p_j^{e_j}$.
\end{theorem}

Note that Filaseta-Schinzel theorem can be applied repeatedly to remove, at each step, 
a prime factor from~$b$ for which $\Phi_b(y)\mid P(y)$,
as long as the prime factors of~$b$ satisfy the condition~(\ref{eq-FS-condition}).
The polynomial to which we will apply Filaseta-Schinzel theorem is
\begin{equation}
\label{eq-q-polynomial}
Q_{S_t}(y)=(y-1)P^*_{S_t}(y) = y^{4t+3} - y^{3t+2} + y^{3t+1} - y^{2t+2} + y^{2t+1} - y^{t+2} + y^{t+1} - 1
\end{equation}
for which $N=8$.
Note that the condition $\Phi_b(y)\mid P^*_{S_t}(y)$ for $b\geq 3$ is equivalent to $\Phi_b(y)\mid Q_{S_t}(y)$,
since $\Phi_b(y)$ is irreducible by definition and thus $\gcd(\Phi_b(y), y-1)=\gcd(\Phi_b(y), \Phi_1(y))=1$ whenever $b\neq 1$.

Before we delve into details,
let us briefly outline the proof of Theorem~\ref{th-main}:
we assume that $\Phi_b(y)\mid Q_{S_t}(y)$ for some $b\geq 3$.
In Lemma~\ref{le-square-free} we obtain a contradiction in the case when $b$ is not square-free, 
i.e., there exists a prime~$p$ such that $p^2\mid b$.
If a square-free~$b$ is not divisible by either of 3, 5 and 7,
the repeated application of Filaseta-Schinzel theorem then implies that
for some prime $p>7$ either $\Phi_p(y)\mid Q_{S_t}(y)$ or $\Phi_{2p}(y)\mid Q_{S_t}(y)$.
Contradiction in these cases is obtained in Lemma~\ref{le-p-2p}.
On the other hand, if a square-free~$b$ is divisible by either of 3, 5 and 7,
then Filaseta-Schinzel theorem can be applied repeatedly to remove the prime factors of~$b$ that are larger than~7,
concluding that $\Phi_{b'}(y)\mid Q_{S_t}(y)$ 
for some $b'$ that is a product of a subset of $\{2,3,5,7\}$
(where 5 and 7 do not both belong to this subset,
 as otherwise Filaseta-Schinzel theorem can be applied again to remove one of these two factors).
This implies that $b'\in\{3,5,6,7,10,14,15,21,30,42\}$ and 
contradiction in these cases is obtained in Lemma~\ref{le-final-b}.
Taken together, these contradictions imply that 
$P^*_{S_t}(y)$ is not divisible by $\Phi_b(y)$ for any $b\geq 3$,
i.e., that $\Circ(n,S_t)$ is a nut graph for each even $n\geq 4t+4$.

Let us now start with an auxiliary claim on the set of exponents appearing in~$Q_{S_t}(y)$,
that will be needed in the proofs of both Lemmas \ref{le-square-free} and~\ref{le-p-2p}.

\begin{claim}
\label{le-unique-rem}
For each prime $p\geq 5$ and each odd $t\geq 3$ such that $t\not\equiv_{10} 1$, 
there always exists an element of the set
$$
I_{S_t} = \{4t+3, 3t+2, 3t+1, 2t+2, 2t+1, t+2, t+1, 0\}
$$
whose remainder modulo~$p$ is unique within this set.
\end{claim}

\begin{proof}
Let us first cover a few special cases:

\medskip\noindent
{\em Case} $t \equiv_p 1$.\quad Then
\begin{align*}
    4t+3 &\equiv_p 7, &
    3t+2 &\equiv_p 5, & 
    3t+1 &\equiv_p 4, &
    2t+2 &\equiv_p 4, \\
    2t+1 &\equiv_p 3, & 
      t+2 &\equiv_p 3, &
      t+1 &\equiv_p 2, &
         0 &\equiv_p 0.
\end{align*}
From $p\mid t-1$ and $10\nmid t-1$ we see that $p\neq 5$, i.e., $p\geq 7$.
Then $3t+2$ and $t+1$ have unique remainders modulo~$p$.

\medskip\noindent
{\em Case} $t \equiv_p 0$.\quad Then
\begin{align*}
    4t+3 &\equiv_p 3, &
    3t+2 &\equiv_p 2, &
    3t+1 &\equiv_p 1, &
    2t+2 &\equiv_p 2, \\
    2t+1 &\equiv_p 1, &
      t+2 &\equiv_p 2, &
      t+1 &\equiv_p 1, &
         0 &\equiv_p 0. 
\end{align*}
In this case, $4t+3$ and 0 have unique remainders modulo~$p$.

\medskip\noindent
{\em Case} $t \equiv_p -1$.\quad Then
\begin{align*}
    4t+3 &\equiv_p -1, &
    3t+2 &\equiv_p -1, &
    3t+1 &\equiv_p -2, &
    2t+2 &\equiv_p  0, \\
    2t+1 &\equiv_p -1, &
      t+2 &\equiv_p  1, &
      t+1 &\equiv_p  0, &
         0 &\equiv_p  0.
\end{align*}
Here $3t+1$ and $t+2$ have unique remainders modulo~$p$.

\medskip\noindent
{\em Case} $t \equiv_p -2$.\quad Then
\begin{align*}
    4t+3 &\equiv_p -5, &
    3t+2 &\equiv_p -4, &
    3t+1 &\equiv_p -5, &
    2t+2 &\equiv_p -2, \\
    2t+1 &\equiv_p -3, &
      t+2 &\equiv_p  0, &
      t+1 &\equiv_p -1, &
         0 &\equiv_p  0.
\end{align*}
In this case each of $3t+2$, $2t+2$, $2t+1$ and $t+1$ has a unique remainder modulo~$p$.

\medskip\noindent
{\em Case} $t\not\equiv_p 1, 0, -1, -2$.\quad
If there would not be a unique remainder modulo~$p$ in this case,
then each remainder will appear at least twice.
and since $|I_{S_t}|=8$, there would be at most four distinct remainders.
As in this case 
$$
0,\ t+1,\ t+2,\ 2t+2
$$ 
all have different remainders modulo~$p$,
we conclude that $I_{S_t}$ then consists of four pairs of elements with common remainder modulo~$p$.

Since
$$
2t+1\not\equiv_p t+1,\ t+2,\ 2t+2
$$
we must have $2t+1\equiv_p 0$. 
This implies 
\begin{align*}
    3t+2 &\equiv_p t+1, &
    4t+3 &\equiv_p 2t+2,
\end{align*}
which leaves as the last pair
$$
3t+1\equiv_p t+2.
$$
Hence $2t-1\equiv_p 0$, which together with $2t+1\equiv_p 0$ yields $2\equiv_p 0$, which is a contradiction since $p\geq 5$.
Thus some element of~$I_{S_t}$ must have a unique remainder modulo~$p$ in this case as well.
\end{proof}

\begin{lemma}
\label{le-square-free}
Let $t\geq 3$ be an odd integer such that $t\not\equiv_{10} 1$ and $t\not\equiv_{18} 15$.
If there exists a prime $p$ such that $p^2\mid b$, then $Q_{S_t}(y)$ is not divisible by $\Phi_b(y)$.
\end{lemma}

\begin{proof}
From the relation $\Phi_{np}(y) = \Phi_n(y^p)$ which holds when $p\mid n$ (see, e.g., \cite[p.~160]{Nage}),
we see that $\Phi_b(y)=\Phi_{b/p}(y^p)$ is a polynomial in $y^p$.

Suppose now that $Q_{S_t}(y) = \Phi_b(y) V(y)$ 
for some polynomial $V(y)=\sum_{i=0}^d v_i y^i\in\Z[y]$.
Then we can write
$$
Q_{S_t}(y) = \sum_{j=0}^{p-1} \sum_{\substack{0\leq i\leq d \\ i\equiv_p j}} \Phi_{b/p}(y^p) v_i y^i.
$$
For $j=0,\dots,p-1$, the inner sum 
$$
Q_{S_t}^{(j)}(y) = \sum_{\substack{0\leq i\leq d \\ i\equiv_p j}} \Phi_{b/p}(y^p) v_i y^i 
                            = \Phi_b(y)\sum_{\substack{0\leq i\leq d \\ i\equiv_p j}} v_i y^i
$$
is comprised of all terms of $Q_{S_t}(y)$ whose exponents are congruent to $j\bmod p$,
and it is divisible by~$\Phi_b(y)$.
We obtain a contradiction independently in each of the following cases.

\bigskip\noindent            
\emph{Case }$p = 2$.\quad
In this case $2 \mid 0,\ t+1,\ 2t+2,\ 3t+1$ and $2 \nmid t+2,\ 2t+1,\ 3t+2,\ 4t+3$. 
Thus
    \begin{align*}
       \Phi_b(y) \mid Q_{S_t}^{(0)}(y) &= y^{3t+1} - y^{2t+2} + y^{t+1} - 1,\\
       \Phi_b(y) \mid Q_{S_t}^{(1)}(y) &= y^{4t+3} - y^{3t+2} + y^{2t+1} - y^{t+2} \\
                                                           &= y^{t+2}\left(y^{3t+1} - y^{2t} + y^{t-1} - 1\right).
    \end{align*}
As $\Phi_b(y)$ is irreducible, we have $\Phi_b(y) \mid y^{3t+1} - y^{2t} + y^{t-1} - 1$. 
Subtracting, we obtain
\begin{align*}
        \Phi_b(y) &\mid (y^{3t+1} - y^{2t} + y^{t-1} - 1) - (y^{3t+1} - y^{2t+2} + y^{t+1} - 1) \\
                      &= y^{2t+2} - y^{2t} - y^{t+1} + y^{t-1} \\
                      &= y^{t-1}\left(y^{t+3} - y^{t+1} - y^{2} + 1\right).
\end{align*}
Hence 
$$
\Phi_b(y)\mid y^{t+3} - y^{t+1} - y^{2} + 1 = (y^{t+1} - 1)(y-1)(y+1).
$$
From $4\mid b$ we have $b\geq 4$, and so
$$
\Phi_b(y)\mid y^{t+1}-1.
$$
Subtracting again,
$$
\Phi_b(y)\mid (y^{3t+1} - y^{2t+2} + y^{t+1} - 1) - (y^{t+1} - 1) 
                   = y^{3t+1} - y^{2t+2} 
                   = y^{2t}\left(y^{t+1}-y^2\right).
$$
Hence $\Phi_b(y)\mid y^{t+1}-y^2$, which together with $\Phi_b(y)\mid y^{t+1}-1$, 
implies $\Phi_b(y) \mid y^2 - 1=(y-1)(y+1)$, which is a contradiction for $b\geq 4$.

\bigskip\noindent
\emph{Case }$p = 3$.\quad
In this case $9\mid b$ (in particular, $b\geq 9$) and we have three subcases depending on $t \bmod 3$.
The appropriate modular values are shown in Table~\ref{tb-exponents-mod-3}.
    
    \begin{table}[h!t]
    {\footnotesize
    \begin{center}
    \begin{tabular}{rccc}
    \toprule                & $t \equiv_3 0$ & $t \equiv_3 1$ & $t \equiv_3 2$ \\
    \midrule
    $(4t+3) \bmod 3$ & $0$                 & $1$                 & $2$\\
    $(3t+2) \bmod 3$ & $2$                 & $2$                 & $2$\\
    $(3t+1) \bmod 3$ & $1$                 & $1$                 & $1$\\
    $(2t+2) \bmod 3$ & $2$                 & $1$                 & $0$\\
    $(2t+1) \bmod 3$ & $1$                 & $0$                 & $2$\\
    $(t+2)   \bmod 3$ & $2$                 & $0$                 & $1$\\
    $(t+1)   \bmod 3$ & $1$                 & $2$                 & $0$\\
    $0         \bmod 3$ & $0$                 & $0$                 & $0$\\
    \bottomrule
    \end{tabular}
    \end{center}
    \caption{Exponents of terms of $Q_{S_t}(y)$ modulo 3.}
    \label{tb-exponents-mod-3}
    }
    \end{table}

\bigskip\noindent
\emph{Subcase }$t\equiv_3 0$.\quad
In this subcase 
\begin{align}
\label{eq-phi-y4t+3-1}
\Phi_b(y) &\mid Q_{S_t}^{(0)}(y) = y^{4t+3}-1, \\
\nonumber
\Phi_b(y) &\mid Q_{S_t}^{(1)}(y)=y^{3t+1} + y^{2t+1} + y^{t+1}=y^{t+1}\left(y^{2t} + y^{t} + 1\right),
\end{align}
and thus $\Phi_b(y)\mid y^{2t} + y^{t} + 1$.
Then also 
\begin{equation}
\label{eq-phi-y3t-1}
\Phi_b(y) \mid (y^{2t} + y^{t} + 1)(y^t - 1)=y^{3t} - 1.
\end{equation}
From (\ref{eq-phi-y4t+3-1}) and~(\ref{eq-phi-y3t-1}) we further have
\begin{align}
\label{eq-phi-y12t+9-1}
\Phi_b(y) &\mid (y^{4t+3})^3-1 = y^{12t+9}-1, \\
\label{eq-phi-y12t-1}
\Phi_b(y) &\mid (y^{3t})^4-1 = y^{12t}-1.
\end{align}
From (\ref{eq-phi-y12t-1}) follows $\Phi_b(y)\mid y^9\left(y^{12t}-1\right)$,
which together with (\ref{eq-phi-y12t+9-1}) yields
\begin{equation}
\label{eq-phi-y9-1}
\Phi_b(y)\mid y^9-1.
\end{equation}
Due to $t \not\equiv_{18} 15$, either $t \equiv_9 0$ or $t \equiv_9 3$. 

If $t \equiv_9 0$ then $4t/9$ is an integer, so from~(\ref{eq-phi-y9-1}) we have
$$
\Phi_b(y) \mid (y^9)^{\frac{4t}{9}} - 1 = y^{4t} - 1,
$$
and also $\Phi_b(y)\mid y^3(y^{4t}-1)$.
Subtracting this from~(\ref{eq-phi-y4t+3-1}) we arrive at
$$
\Phi_b(y) \mid y^3-1=(y-1)(y^2+y+1),
$$
which is a contradiction due to $b\geq 9$.

Likewise, if $t \equiv_9 3$ then $4(t-3)/9$ is an integer, so from~(\ref{eq-phi-y9-1}) we have
$$
\Phi_b(y) \mid (y^9)^{\frac{4(t-3)}{9} + 1} - 1 = y^{4t-3} - 1,
$$
and also $\Phi_b(y)\mid y^6(y^{4t-3}-1)$. 
Subtracting this from~(\ref{eq-phi-y4t+3-1}) we arrive at
$$
\Phi_b(y) \mid y^6-1.
$$
The iterative method based on the lower bound~(\ref{eq-phi-lower-bound}) yields that
$$
\phi(b)\leq 6 \quad\Rightarrow\quad b\leq 25.
$$
Together with $9\mid b$ this implies that either $b=9$ and $b=18$.
However $\Phi_9(y)=y^6+y^3+1$ and $\Phi_{18}(y)=y^6-y^3+1$,
which both contradict $\Phi_b(y)\mid y^6-1$.

\bigskip\noindent
\emph{Subcase }$t \equiv_3 1$.\quad
In this subcase 
\begin{align}
\label{eq-phi-y2t+1-yt+2-1}
\Phi_b(y) &\mid Q_{S_t}^{(0)}(y) = y^{2t+1} - y^{t+2} - 1, \\
\nonumber
\Phi_b(y) &\mid Q_{S_t}^{(2)}(y) = -y^{3t+2}+y^{t+1} = -y^{t+1}\left(y^{2t+1}-1\right),
\end{align}
and thus $\Phi_b(y) \mid y^{2t+1} - 1$. 
Together with~(\ref{eq-phi-y2t+1-yt+2-1}),
this implies $\Phi_b(y)\mid y^{t+2}$ after subtracting, 
which is a contradiction.

\bigskip\noindent
\emph{Subcase }$t\equiv_3 2$.\quad
In this subcase
$$
\Phi_b(y) \mid Q_{S_t}^{(0)}(y) = -y^{2t+2} + y^{t+1} - 1,
$$
from where $\Phi_b(y)\mid (y^{2t+2} - y^{t+1} + 1)(y^{t+1}+1)=y^{3t+3} + 1$, and
$$
\Phi_b(y) \mid Q_{S_t}^{(1)}(y) = y^{3t+1} - y^{t+2} = y^{t+2}(y^{2t-1}-1),
$$
from where $\Phi_b(y)\mid y^{2t-1}-1$.
From these two implications we further have
\begin{align*}
\Phi_b(y) &\mid (y^{3t+3}+1)(y^{3t+3}-1) = y^{6t+6}-1, \\
\Phi_b(y) &\mid (y^{2t-1})^3-1 = y^{6t-3}-1,
\end{align*}
and consequently $\Phi_b(y)\mid (y^{6t+6}-1)-y^9(y^{6t-3}-1) = y^9-1$.
This means that each root~$z$ of~$\Phi_b(y)$ has the form $z=e^{2i\pi k/9}$ for some $k\in\{0,\dots,8\}$.
Due to $t+1=3t'$ for some integer~$t'$,
we have $z^{3t+3}=e^{2i\pi k(3t+3)/9}=e^{2i\pi kt'}=1$.
Since this holds for each root of~$\Phi_b(y)$,
it follows that $\Phi_b(y)\mid y^{3t+3}-1$.
Together with $\Phi_b(y)\mid y^{3t+3}+1$,
this implies $\Phi_b(y)\mid 2$, which is clearly impossible.

\bigskip\noindent
\emph{Case }$p \ge 5$.\quad
In this case Claim~\ref{le-unique-rem} implies that
there exists an element~$f$ of~$I_{S_t}$ with a unique remainder modulo~$p$ within~$I_{S_t}$.
Then $\Phi_b(y)\mid Q_{S_t}^{(f\bmod p)}(y) = \pm y^f$, which is a contradiction, 
as $\Phi_b(y)$ is irreducible with $\deg \Phi_b(y)\geq p(p-1)>1$.
\end{proof}

Before we move on with the next two lemmas, 
note that $\Phi_b(y)\mid y^b-1$ implies that for any integer~$c$
$$
y^c\equiv_{\Phi_b(y)} y^{c\bmod b}.
$$
Hence $\Phi_b(y)\mid Q_{S_t}(y)$ is equivalent to
\begin{align*}
\Phi_b(y) \mid Q_{S_t}^{\!\bmod b}(y) &=
       y^{(4t + 3) \bmod b} - y^{(3t + 2) \bmod b} + y^{(3t + 1) \bmod b} - y^{(2t + 2) \bmod b}\\
    &+ y^{(2t + 1) \bmod b} - y^{(t + 2) \bmod b} + y^{(t + 1) \bmod b} - 1.
\end{align*}    
    
\begin{lemma}
\label{le-p-2p}
For odd $t\geq 3$ and a prime $p\geq 11$,
$Q_{S_t}(y)$ is not divisible by either $\Phi_p(y)$ or $\Phi_{2p}(y)$.
\end{lemma}

\begin{proof}
We know that
    \begin{align*}
        \Phi_{p}(y) &= y^{p-1} + y^{p-2} + \cdots + y^2 + y + 1, \\
        \Phi_{2p}(y) &= y^{p-1} - y^{p-2} + \cdots + y^2 - y + 1.
    \end{align*}
If we assume that $\Phi_p(y)\mid Q_{S_t}(y)$, then also $\Phi_p(y)\mid Q_{S_t}^{\!\bmod p}(y)$.
Since $\deg\Phi_p(y)=p-1\geq\deg Q_{S_t}^{\!\bmod p}(y)$ and 
$\Phi_p(y)$ has more nonzero terms ($p$) than $Q_{S_t}^{\!\bmod p}(y)$ (at most eight),
the case $Q_{S_t}^{{\!\bmod p}}(y) = d\Phi_p(y)$ for some nonzero constant~$d$ is impossible.
Hence it must be $Q_{S_t}^{\!\bmod p}(y)\equiv 0$.
However, Claim~\ref{le-unique-rem} implies the existence of an element $f\in I_{S_t}$ with a unique remainder modulo~$p$.
Thus $Q_{S_t}^{\!\bmod p}(y)$ has the coefficient of $y^{f\bmod p}$ equal to $\pm 1$,
and so it cannot be a zero polynomial, a contradiction.
    
On the other hand, from $(y+1)\Phi_{2p}(y)=y^p+1$, 
we have $y^p\equiv_{\Phi_{2p}(y)} -1$ and thus for any integer~$c$
$$
y^c\equiv_{\Phi_{2p}(y)} (-1)^{\lfloor c/p\rfloor} y^{c\bmod p}.
$$
Hence assuming $\Phi_{2p}(y)\mid Q_{S_t}(y)$ implies
\begin{align*}
\Phi_{2p}(y)\mid Q_{S_t}^{-\!\bmod p}(y) 
                   &= (-1)^{\lfloor\frac{4t+3}{p}\rfloor}y^{(4t + 3) \bmod p} 
                      - (-1)^{\lfloor\frac{3t+2}{p}\rfloor}y^{(3t + 2) \bmod p} \\
                   &+ (-1)^{\lfloor\frac{3t+1}{p}\rfloor}y^{(3t + 1) \bmod p}
                      - (-1)^{\lfloor\frac{2t+2}{p}\rfloor}y^{(2t + 2) \bmod p} \\
                   &+ (-1)^{\lfloor\frac{2t+1}{p}\rfloor}y^{(2t + 1) \bmod p}
                      - (-1)^{\lfloor\frac{t+2}  {p}\rfloor}y^{(t + 2)   \bmod p} \\
                   &+ (-1)^{\lfloor\frac{t+1}  {p}\rfloor}y^{(t + 1)   \bmod p} 
                      - 1.
\end{align*}
Since $\deg\Phi_{2p}(y)=p-1\geq\deg Q_{S_t}^{-\!\bmod p}(y)$,
the only feasible cases are either $\Q_{S_t}^{-\!\bmod p}(y) = d \Phi_{2p}(y)$ for some constant~$d$, 
or $Q_{S_t}^{-\!\bmod p}(y)\equiv 0$.
The former cases is impossible, 
as $d \Phi_{2p}(y)$ has more nonzero terms that $Q_{S_t}^{-\!\bmod p}(y)$.
The latter case is also impossible,
as Claim~\ref{le-unique-rem} implies the existence of $f\in I_{S_t}$ with a unique remainder modulo~$p$ in $I_{S_t}$,
so that the coefficient of $y^{f\bmod p}$ in $Q_{S_t}^{-\!\bmod p}(y)$ is equal to~$\pm 1$, and not zero.
\end{proof}

\begin{lemma}
\label{le-final-b}
Let $t\not\equiv_{10} 1$.
For each $b\in\{3,5,6,7,10,14,15,21,30,42\}$,
the polynomial $Q_{S_t}(y)$ is not divisible by $\Phi_{b}(y)$.
\end{lemma}

\begin{proof}
Note that the lemma implies that $b\in\{3,5,6,7,10,14,15,21,30,42\}$.
Assuming that, on the contrary, $\Phi_b(y)\mid Q_{S_t}(y)$ implies $\Phi_b(y)\mid Q_{S_t}^{\!\bmod b}(y)$.
Since the exponents of $Q_{S_t}^{\!\bmod b}(y)$ range between 0 and $b-1$,
and depend solely on the values of $b$ and~$t\bmod b$,
we can directly compute the remainders after dividing $Q_{S_t}^{\!\bmod b}(y)$ by $\Phi_b(y)$
for each $b\in\{3,5,6,7,10,14,15,21,30,42\}$ and each $t\bmod b\in\{0,\dots,b-1\}$.
These remainders are listed in Appendix~\ref{sc-appendix}.
As it can be seen from this appendix,
$\Phi_b(y)\mid Q_{S_t}^{\!\bmod b}(y)$ holds only if $b=10$ and $t\equiv_{10} 1$,
which is excluded by the assumption.
\end{proof}

Now we are in position to complete the proof of the main theorem.

\bigskip\noindent
{\em Proof of Theorem~\ref{th-main}}.\quad
Suppose, on the contrary, that 
for some odd $t\geq 3$, such that $t\not\equiv_{10}1$ and $t\not\equiv_{18} 15$, and some even $n\geq 4t+4$, 
the circulant graph $\Circ(n,S_t)$ is not a nut graph.
Then $\Phi_b(y)\mid P^*_{S_t}(y)$ for some $b\geq 3$,
which is equivalent to $\Phi_b(y)\mid Q_{S_t}(y)=(y-1)P^*_{S_t}(y)$.

By Lemma~\ref{le-square-free}, such $b$ must be square-free.
Let $b=p_1\cdots p_k$ with $2\leq p_1<\dots<p_k$ be a prime factor decomposition of~$b$.
Denote $b_j=p_1\cdots p_j$ for $j=1,\dots,k$, so that $b=b_k$.
Observe that we can apply Filaseta-Schinzel theorem repeatedly
for each sufficiently large prime factor of~$b$:
starting with $k'=k$,
as long as $p_{k'}\geq 11$,
Theorem~\ref{th-FS} applied to $\Phi_{b_{k'}}\mid Q_{S_t}(y)$, which has $N=8$ nonzero terms, and 
to the prime $p_{k'}$, which then satisfies condition~(\ref{eq-FS-condition}) by itself,
implies that $\Phi_{b_{k'-1}}\mid Q_{S_t}(y)$ as well.

Suppose now that 3, 5 and 7 are not among the prime factors of~$b$.
Applying Filaseta-Schinzel theorem repeatedly as above implies that
either $\Phi_{p_1}(y)\mid Q_{S_t}(y)$ if $p_1\neq 2$ (hence $p_1\geq 11$),
or $\Phi_{2p_2}(y)\mid Q_{S_t}(y)$ if $p_1=2$ (and so $p_2\geq 11$).
However, neither of these cases is possible by Lemma~\ref{le-p-2p}.

Next, suppose that at least one of 3, 5 and 7 is among the prime factors of~$b$.
Applying Filaseta-Schinzel theorem repeatedly as above implies that
$\Phi_{b_j}(y)\mid Q_{S_t}(y)$ where $j$ is the largest index such that $p_j\leq 7$.
Moreover, if $p_j=7$ and $p_{j-1}=5$,
Theorem~\ref{th-FS} can be applied again to $\Phi_{b_j}(y)\mid Q_{S_t}(y)$ and the primes $p_j=7$ and $p_{j-1}=5$
to conclude that also $\Phi_{b_{j-1}}(y)\mid Q_{S_t}(y)$
(and in such case we decrease~$j$ by one).
The resulting~$b_j$ belongs to the set
$$
\{3,5,7,2\cdot3,2\cdot5,2\cdot7,3\cdot5,3\cdot7,2\cdot3\cdot5,2\cdot3\cdot7\},
$$
but $\Phi_{b_j}(y)\mid Q_{S_t}(y)$ then contradicts Lemma~\ref{le-final-b}.
\hfill\qed

\section{Orders and degrees of circulant nut graphs}
\label{sc-other-generators}

Fowler et al.~\cite[Question 11]{FGGPS} asked to characterize 
the pairs $(n,d)$ satisfying requirements of Theorem~\ref{th-vtnut},
for which a $d$-regular vertex-transitive nut graph of order~$n$ exists.
In this section we suggest the answer to this question 
in the case of circulant nut graphs through further computational study.

Lemma~\ref{le-equally-many} tells us that
circulant nut graphs are necessarily $4t$-regular, 
and being vertex-transitive,
they have an even order $n\geq 4t+4$ by Theorem~\ref{th-vtnut}.
The computational study from the previous section shows that for most \emph{odd} values of~$t\leq1300$
there actually exists a single generator set~$S_t=\{1,\dots,2t+1\}\setminus\{t\}$ 
that yields a $4t$-regular circulant nut graph for any even order $n\geq 4t+4$.
In this section we perform a computational study of generator sets of $4t$-regular circulant nut graphs 
for the remaining small values of~$t$: even ones, 
and those odd $t$ such that either $t\equiv1\pmod{10}$ or $t\equiv15\pmod{18}$.
Our findings suggest that for each $t\geq 3$
there exists a single generator set that produces $4t$-regular circulant nut graphs for all feasible values of~$n$.

We start with a result that follows from Theorem~\ref{th-consecutive}.
Although it is applicable to arbitrary values of~$t$, we use it here to settle the case $t=1$.
\begin{theorem}
\label{th-2nd}
Let $t\in\N$. If $n\geq 4t+4$ is an even integer such that $\gcd(\frac n2,t)=1$, 
then there exists a $4t$-regular circulant nut graph of order~$n$.
\end{theorem}

\begin{proof}
We divide the proof in two cases depending on whether $n$ is divisible by 4 or not.

\emph{Case }$n\equiv0\pmod4$:\quad
In this case $\Circ(n,\{x,x+1,\dots,x+2t-1\})$ is a nut graph for $x=\frac n4-t$.
First, $2x+4t=\frac n2+2t<n$ since $2t+2\leq\frac n2$ by the assumption.
Further, $2x+2t-1=\frac n2-1$ is relatively prime to $\frac n2$.
As $\gcd(\frac n2, t)=1$ by the assumption,
from Theorem~\ref{th-consecutive} we conclude that $\Circ(n, \{x,x+1,\dots,x+2t-1\})$ is indeed a nut graph.

\emph{Case }$n\equiv2\pmod4$:\quad
In this case we have $n\geq4t+6$ because $n\equiv2\pmod 4$, so that $x=\frac{n-2}4-t\in\N$.
We have $2x+4t=\frac{n-2}2+2t<n$ since $2t<\frac n2$.
Further, $2x+2t-1=\frac n2-2$.
Since both $\frac n2$ and $\frac n2-2$ are odd in this case, we have $\gcd(\frac n2, 2x+2t-1)=1$.
As $\gcd(\frac n2, t)=1$ by the assumption,
from Theorem~\ref{th-consecutive} we conclude that $\Circ(n, \{x,x+1,\dots,x+2t-1\})$ is again a nut graph.
\end{proof}

For $t=1$ we necessarily have $\gcd(\frac n2,t)=1$ for any even~$n$,
and so we have the following result.
\begin{proposition}
\label{pr-4reg}
There exists a $4$-regular circulant nut graph for each even order~$n\geq 8$.
\end{proposition}

The following lemma shows that for even~$t$, 
regardless of Theorem~\ref{th-vtnut},
one cannot have a $4t$-regular circulant nut graph of order $4t+4$.

\begin{lemma}
\label{le-even-t}
For even~$t$, a $4t$-regular circulant nut graph has order at least $4t+6$.
\end{lemma}

\begin{proof}
Suppose that $G=\Circ(n,S)$, $n=4t+4$, is a $4t$-regular nut graph for some generator set~$S$.
From Lemma~\ref{le-equally-many} we conclude that
$S$ is a $2t$-element subset of $\{1,\dots,2t+1\}$,
hence $S=\{1,\dots,2t+1\}\setminus\{p\}$.
Here $p$ must be odd in order for $S$ to have equally many even and odd elements.

For $\omega=e^{i\frac{2\pi}n}$ we have $\omega^\frac n4=i$, so that
$$
P(\omega^\frac{n}{4}) 
= \left(i + \frac 1i \right)+\left(i^2 + \frac{1}{i^2}\right) + \cdots + \left(i^{2t+1} + \frac{1}{i^{2t+1}}\right) 
 - \left(i^p + \frac{1}{i^p}\right)
$$
From
$$
i^b + \frac 1{i^b} 
= \left\{\begin{array}{rl}
    0, & \text{if $b$ is odd}\\
   -2, & \text{if }b \equiv 2 \pmod 4\\
   +2, & \text{if }b \equiv 0 \pmod 4
   \end{array}\right.
$$
and the fact that $t$ is even, we have
$$
P(\omega^\frac n4) = (0 - 2 + 0 + 2) + \cdots + (0 - 2 + 0 + 2) + 0 - 0 = 0
$$
Since $P(\omega^j)=0$ for some $j\neq n/2$,
Lemma~\ref{le-even-order} implies that $G$ is not a nut graph.
\end{proof}

Thus for $t=2$,
an $8$-regular circulant nut graph has order at least~14,
and Theorem~\ref{th-2nd} yields $\Circ(14,\{1,2,3,4\})$ as an example of such nut graph.
Further,
it is straightforward to show that there does not exist an $8$-regular circulant nut graph of order~$16$.
By Lemma~\ref{le-equally-many}
the generator set of such hypothetical graph has to contain two even and two odd integers from the set $\{1, 2, 3, 4, 5, 6, 7\}$. 
There are ${4\choose 2}{3\choose 2}$ generator sets with this property:
\begin{align*}
\{1,3,2,4\}, && \{1,5,2,4\}, && \{1,7,2,4\}, && \{3,5,2,4\}, && \{3,7,2,4\}, && \{5,7,2,4\}, \\
\{1,3,2,6\}, && \{1,5,2,6\}, && \{1,7,2,6\}, && \{3,5,2,6\}, && \{3,7,2,6\}, && \{5,7,2,6\}, \\
\{1,3,4,6\}, && \{1,5,4,6\}, && \{1,7,4,6\}, && \{3,5,4,6\}, && \{3,7,4,6\}, && \{5,7,4,6\}.
\end{align*}
However, none of them yields a nut graph of order~16,
as the multiplicity of eigenvalue zero in each of them is either 3 or 5 or~9.
Finally, for the generator set $T=\{3,4,5,8\}$
we have 
$$
P^*_T(y) = y^8 P_T(y) = y^{16} + y^{13} + y^{12} + y^{11} + y^5 + y^4 + y^3 + 1.
$$
The iterative method based on the bound~(\ref{eq-phi-lower-bound}) quickly yields 
that $\phi(b)\leq16 \Rightarrow b\leq 68$.
Dividing~$P^*_T(y)$ by~$\Phi_b(t)$ for those $3\leq b\leq 68$ for which actually $\phi(b)\leq 16$,
reveals that $P^*_T(y)$ is not divisible by $\Phi_b(y)$ for any $b\geq 3$, 
and consequently that, apart from~$-1$, 
$P^*_T(y)$ does not have roots of the form $e^{i\frac{2\pi a}b}$ for any $a,b\in \mathbb{Z}$, $\gcd(a,b)=1$.
Hence $\Circ(n,\{3,4,5,8\})$ is a nut graph for any even $n\geq 18$.
We can collect these arguments in the following proposition.
\begin{proposition}
\label{pr-8reg}
There exists an $8$-regular circulant nut graph of order~$n$ 
if and only if $n$ is even and either $n=14$ or $n\geq18$.
\end{proposition}


In the rest of this section,
we will say that the generator set~$S$ is \emph{universal} 
if the circulant graph $\Circ(n, S)$ is a nut graph for each even $n \ge 2\max(S)+2$.
Theorem~\ref{th-main} reveals that the generator set $S_t=\{1,\dots,2t+1\}\setminus\{t\}$ is universal
for all odd $t\geq 3$ such that $t\not\equiv_{10} 1$ and $t\not\equiv_{18} 15$.

Interestingly,
further computational study shows that
for each of the remaining odd values $3\leq t\leq 1300$ such that either $t\equiv_{10}1$ or $t\equiv_{18}15$
there still exists a universal generator set of the form $\{1,\dots,2t+1\}\setminus\{p_t\}$, for some odd~$p_t$.
For some values of~$t$, such as $t=15$ or $t=21$, there exists only a single possibility for~$p_t$,
while for other values, such as $t=11$ or $t=31$, there exists a number of possibilities for~$p_t$.
More information about the possible values of~$p_t$ is contained in the file \emph{oddpt.tex} in~\cite{zenodo},
while in Table~\ref{tb-generators} we list only the smallest possibility for~$p_t$ for $t\leq 120$.

For even values of $t$, $4\leq t\leq 1300$,
universal generator sets exist in abundance in the form $\{1,\dots,2t+2\}\setminus\{q_t,r_t\}$,
for some $q_t$ and~$r_t$ of different parity.
Table~\ref{tb-generators} lists only one possibility for $\{q_t,r_t\}$ for each even $4\leq t\leq 120$,
while further such pairs may be found in the file \emph{evenqtrt.tex} in~\cite{zenodo}.

The existence of universal generator sets in all cases $3\leq t\leq 1300$ leads us 
to conclude this manuscript with the following conjecture that builds upon Theorem~\ref{th-main}.
\begin{conjecture}
\label{co-main-2}
For each odd $t\geq 3$
there exists~$p_t$ such that 
$\Circ(n,\{1,\dots,2t+1\}\setminus\{p_t\})$ is a nut graph for each even $n\geq 4t+4$.
For each even $t\geq 4$
there exist $q_t$ and~$r_t$ such that 
$\Circ(n,\{1,\dots,2t+2\}\setminus\{q_t,r_t\})$ is a nut graph for each even $n\geq 4t+6$.
\end{conjecture}

\begin{table}
\begin{center}
{\scriptsize
\begin{tabular}{rlrlrlrl}
\toprule $t$ & Generator set & $t$ & Generator set & $t$ & Generator set & $t$ & Generator set \\ 
\midrule
   & & 
   31 & $\{1,\dots,63\}\setminus\{5\}$ & 
     61 & $\{1,\dots,123\}\setminus\{5\}$ &  
       91 & $\{1,\dots,183\}\setminus\{5\}$ \\
3 & $S_3$ & 
   33 & $\{1,\dots,67\}\setminus\{27\}$ & 
     63 & $S_{63}$ & 
       93 & $S_{93}$ \\ 
5 & $S_5$ & 
  35 & $S_{35}$ & 
    65 & $S_{65}$ & 
      95 & $S_{95}$ \\ 
7 & $S_7$ & 
  37 & $S_{37}$ & 
    67 & $S_{67}$ & 
      97 & $S_{97}$ \\ 
9 & $S_9$ & 
  39 & $S_{39}$ & 
    69 & $\{1,\dots,139\}\setminus\{9\}$ & 
      99 & $S_{99}$ \\ 
11 & $\{1,\dots,23\}\setminus\{5\}$ & 
  41 & $\{1,\dots,83\}\setminus\{5\}$ & 
    71 & $\{1,\dots,143\}\setminus\{5\}$ & 
      101 & $\{1,\dots,203\}\setminus\{5\}$ \\
13 & $S_{13}$ & 
  43 & $S_{43}$ & 
    73 & $S_{73}$ & 
      103 & $S_{103}$ \\ 
15 & $\{1,\dots,31\}\setminus\{27\}$ & 
  45 & $S_{45}$ & 
    75 & $S_{75}$ & 
      105 & $\{1,\dots,211\}\setminus\{45\}$ \\
17 & $S_{17}$ & 
  47 & $S_{47}$ & 
    77 & $S_{77}$ & 
      107 & $S_{107}$ \\ 
19 & $S_{19}$ & 
  49 & $S_{49}$ & 
    79 & $S_{79}$ & 
      109 & $S_{109}$ \\ 
21 & $\{1,\dots,43\}\setminus\{3\}$ & 
  51 & $\{1,\dots,103\}\setminus\{45\}$ & 
    81 & $\{1,\dots,163\}\setminus\{3\}$ & 
      111 & $\{1,\dots,223\}\setminus\{3\}$ \\
23 & $S_{23}$ & 
  53 & $S_{53}$ & 
    83 & $S_{83}$ & 
      113 & $S_{113}$ \\ 
25 & $S_{25}$ & 
  55 & $S_{55}$ & 
    85 & $S_{85}$ & 
      115 & $S_{115}$ \\ 
27 & $S_{27}$ & 
  57 & $S_{57}$ & 
    87 & $\{1,\dots,175\}\setminus\{27\}$ & 
      117 & $S_{117}$ \\ 
29 & $S_{29}$ & 
  59 & $S_{59}$ & 
    89 & $S_{89}$ & 
      119 & $S_{119}$ \\ 
\midrule
   & & 
   32 & $\{1,\dots,66\}\setminus\{3, 4\}$ &
     62 & $\{1,\dots,126\}\setminus\{1,12\}$ &  
       92 & $\{1,\dots,186\}\setminus\{1,24\}$ \\
4 & $\{1,\dots,10\}\setminus\{4,5\}$ & 
  34 & $\{1,\dots,70\}\setminus\{1,40\}$ & 
    64 & $\{1,\dots,130\}\setminus\{1,8\}$ & 
      94 & $\{1,\dots,190\}\setminus\{1,8\}$ \\
6 & $\{1,\dots,14\}\setminus\{1,12\}$ & 
  36 & $\{1,\dots,74\}\setminus\{1,48\}$& 
    66 & $\{1,\dots,134\}\setminus\{1,48\}$ & 
      96 & $\{1,\dots,194\}\setminus\{1,48\}$ \\
8 & $\{1,\dots,18\}\setminus\{3,4\}$ & 
  38 & $\{1,\dots,78\}\setminus\{1,48\}$ & 
    68 & $\{1,\dots,138\}\setminus\{1,12\}$ & 
      98 & $\{1,\dots,198\}\setminus\{1,24\}$ \\
10 & $\{1,\dots,22\}\setminus\{5,8\}$ & 
  40 & $\{1,\dots,82\}\setminus\{1,4\}$ & 
    70 & $\{1,\dots,142\}\setminus\{1,4\}$ & 
      100 & $\{1,\dots,202\}\setminus\{1,4\}$ \\
12 & $\{1,\dots,26\}\setminus\{3,4\}$ & 
  42 & $\{1,\dots,86\}\setminus\{1,24\}$ & 
    72 & $\{1,\dots,146\}\setminus\{1,12\}$ & 
      102 & $\{1,\dots,206\}\setminus\{1,12\}$ \\
14 & $\{1,\dots,30\}\setminus\{3,16\}$ & 
  44 & $\{1,\dots,90\}\setminus\{1,72\}$ & 
    74 & $\{1,\dots,150\}\setminus\{1,48\}$ & 
      104 & $\{1,\dots,210\}\setminus\{1,12\}$ \\
16 & $\{1,\dots,34\}\setminus\{4,5\}$ & 
  46 & $\{1,\dots,94\}\setminus\{1,8\}$ & 
    76 & $\{1,\dots,154\}\setminus\{1,28\}$ & 
      106 & $\{1,\dots,214\}\setminus\{1,8\}$ \\
18 & $\{1,\dots,38\}\setminus\{8,9\}$ & 
  48 & $\{1,\dots,98\}\setminus\{1,24\}$ & 
    78 & $\{1,\dots,158\}\setminus\{1,36\}$ & 
      108 & $\{1,\dots,218\}\setminus\{1,12\}$ \\
20 & $\{1,\dots,42\}\setminus\{3,8\}$ & 
  50 & $\{1,\dots,102\}\setminus\{1,24\}$ & 
    80 & $\{1,\dots,162\}\setminus\{1,36\}$ & 
      110 & $\{1,\dots,222\}\setminus\{1,60\}$ \\
22 & $\{1,\dots,46\}\setminus\{1,4\}$ & 
  52 & $\{1,\dots,106\}\setminus\{1,4\}$ & 
    82 & $\{1,\dots,166\}\setminus\{1,8\}$ & 
      112 & $\{1,\dots,226\}\setminus\{1,4\}$ \\
24 & $\{1,\dots,50\}\setminus\{1,12\}$ & 
  54 & $\{1,\dots,110\}\setminus\{1,48\}$ & 
    84 & $\{1,\dots,170\}\setminus\{1,48\}$ & 
      114 & $\{1,\dots,230\}\setminus\{1,72\}$ \\
26 & $\{1,\dots,54\}\setminus\{8,9\}$ & 
  56 & $\{1,\dots,114\}\setminus\{1,48\}$ & 
    86 & $\{1,\dots,174\}\setminus\{1,12\}$ & 
      116 & $\{1,\dots,234\}\setminus\{1,72\}$ \\
28 & $\{1,\dots,58\}\setminus\{1,4\}$ & 
  58 & $\{1,\dots,118\}\setminus\{1,32\}$ & 
    88 & $\{1,\dots,178\}\setminus\{1,4\}$ & 
      118 & $\{1,\dots,238\}\setminus\{1,4\}$ \\
30 & $\{1,\dots,62\}\setminus\{3,8\}$ & 
  60 & $\{1,\dots,122\}\setminus\{1,24\}$ & 
    90 & $\{1,\dots,182\}\setminus\{1,24\}$ & 
      120 & $\{1,\dots,242\}\setminus\{1,24\}$ \\
\bottomrule
\end{tabular}
\caption{Examples of universal generator sets for $4t$-regular circulant nut graphs
              for odd and even~$t$, $3\leq t\leq 120$.}      
\label{tb-generators}        
}
\end{center}
\end{table}

\subsection*{Acknowledgements}

We are grateful to the following persons who responded to our question on \emph{Mathoverflow}:
to Fedor Petrov for reteaching us how to use cyclotomic polynomials in Sections 4 and~5,
to Vlad Matei for pointing to a reference~\cite{FiSch} on Filaseta-Schinzel theorem on divisibility by cyclotomic polynomials, and
to Richard Stanley for comments about Galois groups of polynomials.

\appendix
\section{Table of remainders $Q_{S_t}^{\!\bmod b}(y) \bmod{\Phi_b(y)}$ for \\ 
              $b\!\in\!\{3,5,6,7,10,14,15,21,30,42\}$ and all $t\bmod b$}
\label{sc-appendix}              

{\scriptsize
\begin{longtable}{rl}
\toprule $t \bmod 3$ & $Q_{S_t}^{\!\bmod 3}(y) \bmod{\Phi_3(y)}$ \\
\midrule
0 & $6 y+3$ \\
1 & $y-1$ \\
2 & $-y-2$ \\
\midrule $t \bmod 5$ & $Q_{S_t}^{\!\bmod 5}(y) \bmod{\Phi_5(y)}$ \\ 
\midrule
0 & $y^3-3 y^2+3 y-1$ \\
1 & $2 y^2-2$ \\
2 & $y^3+2 y^2+y+1$ \\
3 & $-2 y^3-2 y-1$ \\
4 & $-y^2-2 y-2$ \\
\midrule $t \bmod 6$ & $Q_{S_t}^{\!\bmod 6}(y) \bmod{\Phi_6(y)}$ \\ 
\midrule
0 & $1$ \\
1 & $3y-3$ \\
2 & $-y$ \\
3 & $-3$ \\
4 & $y-1$ \\
5 & $-3y$ \\
\midrule $t \bmod 7$ & $Q_{S_t}^{\!\bmod 7}(y) \bmod{\Phi_7(y)}$ \\
\midrule
0 & $y^3-3 y^2+3 y-1$ \\
1 & $-y^5+y^2$ \\
2 & $2 y^5+y^4+2 y^3+y^2+1$ \\
3 & $-y^5+y^3$ \\
4 & $2 y^5-y^3+y^2-2$ \\
5 & $-2 y^5-2 y^3+y^2-y-3$ \\
6 & $-y^4-y^3-y^2-2 y-2$ \\
\midrule $t \bmod 10$ & $Q_{S_t}^{\!\bmod{10}}(y) \bmod{\Phi_{10}(y)}$ \\ 
\midrule
 0 & $y^3-3 y^2+3 y-1$ \\
 1 & $0$ \\
 2 & $y^3+y-1$ \\
 3 & $2 y^3-2 y^2-1$ \\
 4 & $-2 y^3+3 y^2-2 y$ \\
 5 & $y^3+y^2-y-1$ \\
 6 & $-2 y$ \\
 7 & $-y^3+3 y-3$ \\
 8 & $-3$ \\
 9 & $-2 y^3+y^2-2 y$ \\
\midrule $t \bmod 14$ & $Q_{S_t}^{\!\bmod{14}}(y) \bmod{\Phi_{14}(y)}$ \\ 
\midrule
 0 & $y^3-3 y^2+3 y-1$ \\
 1 & $-y^5+y^2-2$ \\
 2 & $-y^4+y^2-1$ \\
 3 & $-y^5+2 y^4-y^3+2 y-2$ \\
 4 & $2 y^4-y^3+y^2-2 y$ \\
 5 & $2 y^5-2 y^4-y^2+y-1$ \\
 6 & $-2 y^5+3 y^4-3 y^3+3 y^2-2 y$ \\
 7 & $y^3+y^2-y-1$ \\
 8 & $y^5-2 y^4+2 y^3-y^2-2$ \\
 9 & $y^4-2 y^3+y^2-2 y+1$ \\
 10 & $y^5-2 y^4+y^3+2 y-2$ \\
 11 & $2 y^5-2 y^4+3 y^3-3 y^2+2 y-2$ \\
 12 & $-y^2-y-1$ \\
 13 & $-2 y^5+y^4-y^3+y^2-2 y$ \\
\midrule $t \bmod 15$ & $Q_{S_t}^{\!\bmod{15}}(y) \bmod{\Phi_{15}(y)}$ \\ 
\midrule
 0 & $y^3-3 y^2+3 y-1$ \\
 1 & $y^7-y^5+y^2-1$ \\
 2 & $-2 y^6+2 y^5-2 y^4+2 y^3-2 y$ \\
 3 & $y^6-y^5+y^3$ \\
 4 & $y^7-3 y^6+3 y^5+y^4-2 y^3+2 y^2-y-1$ \\
 5 & $y^7-y^5+y^4-y^3+y-2$ \\
 6 & $-y^7-y^6+y^5-2 y$ \\
 7 & $y^6+y^3-y^2+1$ \\
 8 & $y^7+y^5-2 y^3+2 y^2+y-2$ \\
 9 & $y^7-y^6-y^5+y^4-2 y^3+2 y^2-3$ \\
 10 & $-y^7+y^5-y^4-y$ \\
 11 & $y^6-3 y^5+3 y^4-y^2+2 y-2$ \\
 12 & $y^6+y^5-y^4+y^2-y-1$ \\
 13 & $-y^7+2 y^6-3 y^5+y^3-2 y^2+2 y-4$ \\
 14 & $-2 y^7+y^6+y^5-2 y^4+y^3-y^2-2 y+1$ \\
\midrule $t \bmod 21$ & $Q_{S_t}^{\!\bmod{21}}(y) \bmod{\Phi_{21}(y)}$ \\ 
\midrule
 0 & $y^3-3 y^2+3 y-1$ \\
 1 & $y^7-y^5+y^2-1$ \\
 2 & $y^{11}-y^8+y^7-y^6+y^5-y^4+y^3-1$ \\
 3 & $-y^{11}+y^{10}-2 y^8+y^7-y^5+y^4-y-1$ \\
 4 & $-2 y^{10}+2 y^9-y^6+y^5-y^4+y^2-y$ \\
 5 & $y^{10}-y^8-y^7+2 y^6-y^4+2 y^3-y$ \\
 6 & $y^{11}-2 y^{10}+y^9+y^6+y^5-y^4-y^3+2 y^2-y-1$ \\
 7 & $y^{10}-1$ \\
 8 & $y^{11}-2 y^{10}+y^9-y^7-y^5+2 y^4-y^3-2$ \\
 9 & $-2 y^{11}+y^9-y^8+y^6-2 y^4+y^2-y-1$ \\
 10 & $-y^{11}+y^{10}+y^9-y^8+y^6-y^4+y^3-y+1$ \\
 11 & $y^{11}-y^9+y^8+y^7-y^6+y^5+y^4-2 y^3+y^2+y-1$ \\
 12 & $y^{11}+y^8-y^6+y^4-1$ \\
 13 & $y^{11}-2 y^{10}+y^9+2 y^8-4y^7+y^6+y^5-y^4-y^3+2 y^2-3$ \\
 14 & $-y^{10}-y^3-1$ \\
 15 & $-y^{11}+2 y^{10}-y^9-y^5+y^4+y^3-y^2$ \\
 16 & $y^{11}-y^9-y^8+2 y^7-y^5+3 y^4-y^3-y^2+y-1$ \\
 17 & $-y^{11}+y^{10}-y^9+3 y^8-y^7-y^6+y^5-y^3+2 y-3$ \\
 18 & $2 y^{11}-y^{10}-y^9+2 y^8-y^7-y^6+y^5-y^3+y^2-2$ \\
 19 & $-y^{11}+2 y^{10}-3 y^9+y^7-y^6+y^3-3 y^2+y-2$ \\
 20 & $-2 y^{11}+y^{10}+y^9-2 y^8+y^7+y^6-2y^5 -y^4+2 y^3-y^2-2 y+1$ \\
\midrule $t \bmod 30$ & $Q_{S_t}^{\!\bmod{30}}(y) \bmod{\Phi_{30}(y)}$ \\ 
\midrule
 0 & $y^3-3 y^2+3 y-1$ \\
 1 & $y^7-y^5+y^2-1$ \\
 2 & $2 y^7-2 y^4$ \\
 3 & $2 y^7-y^6-y^5-y^3+2 y-2$ \\
 4 & $-y^7-y^6+y^5+y^4-y-1$ \\
 5 & $-y^7+2 y^6-y^5-y^4-y^3+2 y^2-y$ \\
 6 & $-y^7-y^6+y^5+2 y^2-2$ \\
 7 & $-4 y^7-y^6+2 y^5+2 y^4+3 y^3+y^2-5$ \\
 8 & $y^7+2 y^6-y^5-2 y^2-y+2$ \\
 9 & $3 y^7+y^6+y^5-3 y^4-2 y^3-2 y^2+2 y+1$ \\
 10 & $-y^7+y^5+y^4-y-2$ \\
 11 & $4 y^7+y^6-3 y^5-y^4-2 y^3-3 y^2+2 y+2$ \\
 12 & $-y^6-y^5+y^4+y^2-y-1$ \\
 13 & $-y^7+y^5-y^3$ \\
 14 & $-4 y^7-3 y^6+y^5+4 y^4+3 y^3+3 y^2-7$ \\
 15 & $y^3+y^2-y-1$ \\
 16 & $y^7+y^5-2 y^4+2 y^3-y^2-1$ \\
 17 & $-2 y^7+2 y^5+2 y^4-2 y-2$ \\
 18 & $2 y^7+y^6-y^5-2 y^4-y^3$ \\
 19 & $3 y^7+3 y^6-3 y^5-3 y^4-2 y^3-2 y^2+y+3$ \\
 20 & $y^7-y^5-y^4-y^3+y$ \\
 21 & $-5 y^7-y^6+y^5+4 y^4+2 y^3+2 y^2-2 y-4$ \\
 22 & $4 y^7+y^6-2 y^5-2 y^4-3 y^3-y^2+4 y+1$ \\
 23 & $-y^7-2 y^6+3 y^5+2 y^3+y-4$ \\
 24 & $-y^7+y^6+y^5+y^4-2 y-1$ \\
 25 & $y^7-2 y^6+y^5+y^4-y-2$ \\
 26 & $y^6+y^5-y^4-2 y^3-y^2$ \\
 27 & $y^6-y^5-y^4-y^2-y+1$ \\
 28 & $-3 y^7-y^5+2 y^4+y^3+2 y^2-2 y-2$ \\
 29 & $-y^6-y^5+y^3+y^2-1$ \\
\midrule $t \bmod 42$ & $Q_{S_t}^{\!\bmod{42}}(y) \bmod{\Phi_{42}(y)}$ \\ 
\midrule
 0 & $y^3-3 y^2+3 y-1$ \\
 1 & $y^7-y^5+y^2-1$ \\
 2 & $y^{11}-y^8+y^7-y^6+y^5-y^4+y^3-1$ \\
 3 & $-y^{11}+y^{10}+y^7-y^5+y^4-y-1$ \\
 4 & $2 y^9-2 y^7-y^6+y^5+y^4-y^2-y$ \\
 5 & $2 y^{11}-y^{10}-y^8-y^7+2 y^6-y^4-2 y^2+y$ \\
 6 & $-y^{11}+y^9-y^6-y^5+y^4+y^3-y-1$ \\
 7 & $-y^{10}-2 y^9+2 y^8+2 y^2-2 y-1$ \\
 8 & $-y^{11}+y^9-y^7+y^5-y^3$ \\
 9 & $-4 y^{11}+y^9+3 y^8-y^6-2 y^5+2y^4+2 y^3+y^2-y-3$ \\
 10 & $3 y^{11}-y^{10}-y^9-y^8+y^6-y^4-y^3+3 y-1$ \\
 11 & $-3 y^{11}-2 y^{10}+y^9+3 y^8+3 y^7-3 y^6-y^5+y^4+4 y^3+y^2-y-5$ \\
 12 & $y^{11}+2 y^{10}-y^8-2 y^7+y^6+2y^5-y^4-2 y^3+1$ \\
 13 & $3 y^{11}+2 y^{10}-y^9-4 y^8+y^6+3 y^5-y^4-3 y^3-2 y^2+2 y+1$ \\
 14 & $y^{10}-y^3-1$ \\
 15 & $y^{11}-2 y^{10}+y^9-y^5+y^4+y^3-y^2-2$ \\
 16 & $y^{11}+2 y^{10}-y^9-3 y^8+2 y^6+y^5-y^4-3 y^3-y^2+y+1$ \\
 17 & $y^{11}+y^{10}-y^9-y^8-y^7+y^6+y^5-2 y^4-y^3+1$ \\
 18 & $y^{10}-y^9-y^7+y^6-y^5-y^3+y^2$ \\
 19 & $-y^{11}+y^9-y^7-y^6+2 y^4+y^3-y^2-y$ \\
 20 & $-4 y^{11}-3 y^{10}+y^9+4 y^8+3 y^7-y^6-4y^5+y^4+4 y^3+3 y^2-7$ \\
 21 & $y^3+y^2-y-1$ \\
 22 & $y^7+y^5-2 y^4+2 y^3-y^2-1$ \\
 23 & $y^{11}+y^8-y^7-y^6+y^5+y^4-y^3-1$ \\
 24 & $y^{11}-y^{10}+y^7+y^5-y^4-y-1$ \\
 25 & $-2 y^{11}-2 y^{10}+2 y^9+2 y^8+2 y^7-y^6-3y^5+y^4+2 y^3+y^2-y-4$ \\
 26 & $2 y^{11}+y^{10}-2 y^9-y^8+y^7-y^4-2 y^3+y$ \\
 27 & $3 y^{11}+2 y^{10}-y^9-2 y^8-4 y^7+y^6+3 y^5-y^4-3 y^3-2 y^2+y+3$ \\
 28 & $-y^{10}-1$ \\
 29 & $-y^{11}+2 y^{10}-y^9-y^7-y^5+2 y^4-y^3$ \\
 30 & $-2 y^{11}-2 y^{10}+y^9+y^8+2 y^7-y^6-2 y^5+2y^4+2 y^3+y^2-y-3$ \\
 31 & $-3 y^{11}+y^{10}+y^9+y^8-y^6+y^4+y^3+y-3$ \\
 32 & $3 y^{11}+2 y^{10}-y^9-3 y^8-3 y^7+3 y^6+3y^5-y^4-2 y^3-3 y^2+y+3$ \\
 33 & $-y^{11}-2 y^{10}+2 y^9+y^8+2y^7-y^6-y^4+2 y^3-3$ \\
 34 & $-y^{11}+y^9+2 y^8-y^6-y^5+y^4+y^3-2 y-1$ \\
 35 & $y^{10}+2 y^9-2 y^8-y^3-1$ \\
 36 & $y^{11}-y^9+y^5-y^4-y^3+y^2-2$ \\
 37 & $3 y^{11}-y^9-y^8-2 y^7+2 y^6+y^5-3 y^4-y^3-y^2+y+1$ \\
 38 & $-y^{11}-y^{10}+y^9+y^8-y^7-y^6-y^5+2 y^4+y^3-2 y-1$ \\
 39 & $2 y^{11}+y^{10}-3 y^9-2 y^8+y^7+y^6+y^5-2 y^4-3 y^3+y^2+2 y$ \\
 40 & $-3 y^{11}-y^9+2 y^8+y^7-y^6-2 y^5+2 y^4+y^3+3 y^2-y-4$ \\
 41 & $-y^{10}-y^9+y^7+y^6-y^4+y^2-1$ \\
\bottomrule
\end{longtable}
}

\end{document}